\documentclass[11pt]{article}
\usepackage[left=1in,right=1in,top=1in,bottom=1in]{geometry}

\usepackage{amsthm}
\usepackage{hyperref}
    \hypersetup{colorlinks,
        linkcolor={black},
        citecolor={black},
    }
\usepackage{amsmath}

\usepackage{mathrsfs}
\usepackage[nameinlink]{cleveref}
    \crefname{ex}{Example}{Examples}
    \crefname{thm}{Theorem}{Theorems} 
    \crefname{lem}{Lemma}{Lemmas}
    \crefname{prop}{Proposition}{Propositions}
    \crefname{cor}{Corollary}{Corollaries} 
    \crefname{conj}{Conjecture}{Conjectures} 
    \crefname{defn}{Definition}{Definitions}
    \crefname{rmk}{Remark}{Remarks} 
    
	\newtheorem{thm}{Theorem}[section]
	\newtheorem{lem}[thm]{Lemma}

	\newtheorem*{thm*}{Theorem}
	\newtheorem*{cor*}{Corollary}
	\theoremstyle{definition} 
		\newtheorem{defn}[thm]{Definition}
		\newtheorem{ex}[thm]{Example}
    	\newtheorem{rmk}[thm]{Remark}	
    \newcommand{\df}[1]{{\bf\emph{#1}}}			

	\newtheoremstyle{TheoremNum}
        {\topsep}{\topsep}              
        {\itshape}                      
        {}                              
        {\bfseries}                     
        {.}                             
        { }                             
        {\thmname{#1}\thmnote{ \bfseries #3}}
    \theoremstyle{TheoremNum}

\usepackage{tikz, graphicx}
\usepackage[margin=1.5cm]{caption}
\usepackage{subcaption}
	\usetikzlibrary{positioning, arrows, arrows.meta}
	\usepackage{pgfplots}
        \tikzset{%
        fwdrxn/.style={very thick, arrows={-Stealth[length=5pt,width=5pt]}},
        revrxn/.style={very thick, arrows={-Stealth[length=5pt,width=5pt,left]}},
        newt/.style={turq, opacity=0.15}
        }
        \tikzset{near start abs-right/.style={xshift=1cm}}
        \tikzset{near start abs-left/.style={xshift=-3.5cm}}
        \tikzset{near start abs-up/.style={yshift=1.5cm}}
        \tikzset{near start abs-down/.style={yshift=-1cm}}
	
\usepackage{color, xcolor}
    
	\newcommand\blue[1]{{\textcolor{blue}{#1}}}
    \definecolor{viridisyellow}{RGB}{253,231,36}
    \definecolor{viridisyellowpale}{RGB}{239,223,81}
    \definecolor{viridisgreen}{RGB}{121,209,81}
        \definecolor{hlgreen}{RGB}{16,115,16}
    \definecolor{viridisturq}{RGB}{34,167,132}
    \definecolor{viridisblue}{RGB}{64,67,135}
    \definecolor{viridisviolet}{RGB}{68,1,84}

    \definecolor{magmapink}{RGB}{188,81,119}
    \definecolor{pastelpink}{RGB}{253,191,210}
    
	\definecolor{ratecnst}{RGB}{172,172,172}
	
\usepackage{multirow}
\usepackage{array}

\usepackage{parskip}

\usepackage[normalem]{ulem}	

\usepackage{amsfonts, amssymb, amsrefs, mathtools}
\newcommand{\eq}[1]{\begin{align*}#1\end{align*}}
	\newcommand{\eqn}[1]{\begin{align}#1\end{align}}  
\newcommand{\ds}{\displaystyle}	
\newcommand{\st}{\colon}          
  
\newcommand\mc[1]{\mathcal{#1}}

\newcommand\mrm[1]{\mathrm{#1}}
 
\newcommand{\rr}{\ensuremath{\mathbb{R}}}   
 
\newcommand{\zz}{\ensuremath{\mathbb{Z}}}
\newcommand{\nn}{\ensuremath{\mathbb{N}}}

\renewcommand{\epsilon}{\varepsilon}	
\renewcommand{\phi}{\varphi}			
\DeclareMathOperator{\ran}{im}	    	
\DeclareMathOperator{\Span}{span}		
\DeclareMathOperator{\rank}{rank}		    
\newcommand{\kk}{\kappa}
\newcommand{\vv}[1]{{\boldsymbol{#1}}}  
\newcommand{\mm}[1]{\mathbf{#1}}               
\newcommand{\rrp}{\rr_{\geq}}
\newcommand{\rrpp}{\rr_{>}}
\newcommand{\zzp}{\zz_{\geq}}

\newcommand{\xx}{\vv x}
\newcommand{\yy}{\vv y}
\newcommand{\ratecnst}[1]{{\footnotesize{\color{blue}{#1}}}}
\newcommand{\mtxphantom}{\vphantom{\ds\sum}}

\usepackage{enumitem}
\usepackage[version=4]{mhchem}
\usepackage{chemfig}

\usepackage[colorinlistoftodos,prependcaption,textsize=footnotesize,  
    linecolor=orange, bordercolor=orange, backgroundcolor=viridisyellow, 
    textcolor=black]{todonotes}

\newcommand{\WRz}{WR\textsubscript{0}}

\usepackage{algorithm,algorithmic}

\usepackage{authblk}
\title{
    An algorithm for finding weakly reversible deficiency zero realizations of polynomial dynamical systems
}
\author[1,2]{
         Gheorghe Craciun%
}
\author[3]{
        Jiaxin Jin%
}
\author[4]{
        Polly Y. Yu%
}
\affil[1]{\small Department of Mathematics, University of Wisconsin--Madison}
\affil[2]{\small Department of Biomolecular Chemistry, University of Wisconsin--Madison}
\affil[3]{\small Department of Mathematics, The Ohio State University}
\affil[4]{\small NSF-Simons Center for Mathematical and Statistical Analysis of Biology, Harvard University}
\date{} 

\begin{document}

\maketitle

\begin{abstract}
Systems of differential equations with polynomial right-hand sides are very common in applications. On the other hand, their mathematical analysis is very challenging in general, due to the possibility of complex dynamics: multiple basins of attraction, oscillations, and even chaotic dynamics. Even if we restrict our attention to mass-action systems, all of these complex dynamical behaviours are still possible. On the other hand, if a polynomial dynamical system has a {\em weakly reversible deficiency zero (\WRz) realization}, then its dynamics is known to be remarkably simple: oscillations and chaotic dynamics are ruled out and, up to linear conservation laws, there exists a {\em single positive steady state}, which is asymptotically stable. Here we describe an algorithm for finding {\em \WRz }  realizations of polynomial dynamical systems, whenever such realizations exist. 
\end{abstract}

\section{Introduction}
\label{sec:intro} 

By a \emph{polynomial dynamical system} we mean a system of ODEs with polynomial right-hand side, of the form
\begin{equation}\label{eq:poly-intro}
\begin{split}
    \frac{dx_1}{dt} &= p_1(x_1, ..., x_n), \\ 
    \frac{dx_2}{dt} &= p_2(x_1, ..., x_n), \\ 
                    &\qquad \quad \vdots  \\
    \frac{dx_n}{dt} &= p_n(x_1, ..., x_n), \\ 
\end{split}
\end{equation}
where $p_i(x_1,\ldots, x_n) \in \rr[x_1,\ldots, x_n]$. In general, such systems are very difficult to analyze due to nonlinearities and feedback that may give rise to bifurcations, multiple basins of attraction, oscillations, and even chaotic dynamics. The second part of Hilbert's 16th problem (about the number of limit cycles of polynomial dynamical systems in the plane) is still essentially unsolved, even for {\em quadratic} polynomials~\cite{Ilyashenko2002}.  Even the simplest object associated to \eqref{eq:poly-intro}, its steady state set, is central to real algebraic geometry.
 
In terms of applications, polynomial dynamical systems often show up in, for example, chemistry, biology, and population dynamics. In these models, the variable $x_i$ typically represents concentration, population, or another quantity that is strictly positive, so the domain of \eqref{eq:poly-intro} is restricted to the positive orthant.  For example, in an infectious disease model, an infectious individual might covert a susceptible individual; this would contribute a \lq\lq ${}+bxy$\rq\rq\ term to $\frac{dx}{dt}$, where $x$ is the population of susceptible individuals, $y$ the infectious population, and $b > 0$ a parameter measuring the contact rate. Collecting all contributing terms results in an interaction network. An active area of research is to relate the structure of the interaction network to the dynamics generated by it~\cite{Johnston2014, PerezmillanDickensteinShiuConradi2012, ThomsonGunawardena2009, JoshiShiu2015, JoshiShiu2017, BanajiPantea2016, BanajiPantea2018, ShinarFeinberg2010, MinchevaRoussel2007}. 

Conversely, one may start with \eqref{eq:poly-intro} from experimental data, with little or no information on the generating interaction network. One may try to elucidate the underlying interaction network; however, without additional assumptions, a polynomial dynamical system is not uniquely generated by one interaction network, but infinitely many~\cite{CraciunPantea2008}. This lack of identifiability of the underlying network can actually be leveraged to analyze the dynamics: if a network with certain properties can be found to generate \eqref{eq:poly-intro}, then we may be able to immediately infer its dynamical behaviour. 

A class of systems whose dynamics is very well understood is the family of  {\em complex-balanced systems}~\cite{HornJackson1972}, which are also called {\em toric dynamical systems}~\cite{CraciunDickensteinShiuSturmfels2009}. They can never exhibit oscillations or chaotic dynamics, and, up to linear conservation laws, there exists a {\em single positive steady state}, which is locally asymptotically stable~\cite{HornJackson1972}. Moreover, this steady state is conjectured to be a global attractor~\cite{Horn1974_GAC}. 

Not only are the dynamical properties of complex-balanced systems well understood, but also the network and parameter structures that characterize them~\cite{horn1972necessary}. While in general, there are algebraic conditions on the parameters necessary for complex-balancing, the exception to this rule is the case of weakly reversible and deficiency zero (\WRz) networks -- these systems are complex-balanced {\em for any choices of parameters}, in a sense that will be made clear below. This fact is very important in applications, because the exact values of the coefficients in the polynomial right-hand sides of these dynamical systems are often very difficult to estimate accurately in practice.

In this paper we describe an efficient algorithm for determining whether a given polynomial dynamical system admits a \WRz\ realization, and for finding such a realization whenever it exists (see \Cref{algorithm}). Our algorithm does not require solving the differential equation \eqref{eq:poly-intro}, nor does it require solving for its steady state set. Instead, the algorithm, making use of the geometric and log-linear structure of \WRz\ networks, requires as its inputs only the monomials and the matrix of coefficients. If a \WRz\ realization exists, in \Cref{thm:eqm} we provide a bijection between the positive steady state set of \eqref{eq:poly-intro} and the solution to a system of linear equations. 

The paper is organized as follows. In \Cref{sec:bg} we introduce interaction networks as embedded in $\rr^n$ and formalize their relations to polynomial dynamical systems; we also introduce complex-balanced systems, \WRz\ networks, and other relevant notions and results. In \Cref{sec:alg-pf} we describe our algorithm for finding a \WRz\ realization of a given polynomial dynamical system, whose steady state set is studied in \Cref{sec:ss-set}. Our algorithm applies to the case of where the coefficients in the polynomials are unspecified; we consider such systems in \Cref{sec:unspec_coeff}.

\section{Background}
\label{sec:bg}

Throughout this work, we denote by $\rrp^n$ and $\rrpp^n$ the sets of vectors with non-negative and positive entries respectively. Similarly, $\zzp^n$ is the set of vectors with non-negative integer components. Vectors are typically denoted $\xx$, $\yy$, or $\vv w$. We denote by $\dot{\xx}$ the time-derivative $\frac{d\xx}{dt}$. For any $\xx \in \rrpp^n$ and $\yy \in \rr^n$, define the operation $\xx^{\yy} = x_1^{y_1}x_2^{y_2}\cdots x_n^{y_n}$. If $\mm Y = \begin{pmatrix} \yy_1 & \yy_2 & \dots & \yy_n \end{pmatrix}$, then $\xx^{\mm Y} = (\xx^{\yy_1}, \xx^{\yy_2},\ldots, \xx^{\yy_n})^\top$. The support of a vector $\xx \in \rr^n$ is the set of indices $\mrm{supp}(\xx) = \{ i \st x_i \neq 0\}$.

\subsection{Dynamical systems and Euclidean embedded graphs}
\label{sec:intro-EEG}

In this section, we introduce the Euclidean embedded graph (E-graph), a directed graph in $\rr^n$, and explain how a system of differential equations with polynomial right-hand side (a polynomial dynamical system) is defined by it.

\begin{defn}
\label{def:EEG}
    A \df{Euclidean embedded graph} (\df{E-graph}) in $\rr^n$ is a directed graph $(V,E)$, where $V$ is a finite subset of $\rrp^n$, and there are neither self-loops nor isolated vertices. Denote by $V_s$ the set of source vertices.
\end{defn}

Let $V = \{ \yy_1,\yy_2,\ldots, \yy_m\}$. An edge $(\yy_i, \yy_j)$, or $(i,j) \in E$, is also denoted $\yy_i \to \yy_j$. Since vertices are points in $\rr^n$, an edge can be regarded as a bona fide vector between vertices. An \df{edge vector} $\yy_j - \yy_i$ is associated to the edge $\yy_i \to \yy_j$. 

For the purpose of using E-graphs to study polynomial dynamical systems, we assume $V_s \subset \zzp^n$, even though most results stated in this paper hold for $V \subset \rrp^n$. 

The set of vertices $V$ of $(V,E)$ is partitioned by its connected components, which we identify by the subset of vertices that belong to that connected component. If every connected component is strongly connected, i.e., every edge is part of a cycle, then $(V,E)$ is said to be \df{weakly reversible}. 

Two geometric properties of the E-graph will become important to our analysis of polynomial dynamical system. The first is a notion of affine independence within each connected component; the second is a notion of linear independence between connected components.

\begin{defn}
\label{def:aff-indep}
    An E-graph $(V,E)$ has \df{affinely independent connected components} if the vertices in each connected component are affinely independent, i.e., if $\{ \yy_0, \yy_1, \ldots, \yy_r\} \subseteq V$ is a connected component, then the set $\{ \yy_j - \yy_0 \st j=1,2,\ldots, r \}$ is linearly independent.
\end{defn}

\begin{defn}
\label{def:stoichsubsp}
    Let $(V, E)$ be an E-graph. For any $U \subseteq V$, the \df{associated linear subspace of $U$} is $S(U) = \Span \{ \yy_j - \yy_i \st  \yy_i, \,\yy_j \in U\}$.
    The \df{associated linear space}\footnote{In reaction network theory literature, the associated linear space is called the \emph{stoichiometric subspace} of the network.} of $(V,E)$ is 
        \eq{
        S = \Span\{ \vv y_j - \vv y_i \st \yy_i \to \yy_j \in E\} .
    }
\end{defn}

If $U$ defines a connected component of $(V,E)$, then $S(U) \subseteq S$. Indeed, if $V_1$, $V_2, \ldots, V_\ell$ are the connected components, then $S = S(V_1) + S(V_2) +  \cdots + S(V_\ell)$.

Thus far, we have defined an E-graph, and introduced several objects and properties associated to it. We now turn our attention to how such a graph is canonically associated to dynamics, by assigning a positive weight to each edge. 

\begin{defn}
\label{def:MAS} 
    Let $(V,E)$ be an E-graph. For each $\yy_i \to \yy_j \in E$, let $\kk_{ij} > 0$ be its weight, and let $\vv\kk = (\kk_{ij}) \in \rrpp^E$. The \df{associated dynamical system} on $\rrpp^n$ of the weighted E-graph $(V,E,\vv\kk)$ is 
    \eqn{\label{eq:MAS} 
        \frac{d\xx}{dt} = \sum_{(i,j) \in E} \kk_{ij} \xx^{\yy_i} (\yy_j - \yy_i). 
    }
\end{defn}

It is sometimes convenient to refer to $\kk_{ij}$ even though $\yy_i \to \yy_j$ may not be an edge in the network. In such cases, set $\kk_{ij} = 0$.

\begin{rmk}
We defined the domain of \eqref{eq:MAS} to be $\rrpp^n$. Systems of ODEs with polynomial right-hand side do not in general leave $\rrpp^n$ forward-invariant, but if we assume $V \subset \zzp^n$, the positive orthant $\rrpp^n$ is indeed forward-invariant under \eqref{eq:MAS}~\cite{Sontag2001}. 
\end{rmk}

It is clear that the right-hand side of \eqref{eq:MAS} lies in the associated linear space $S$, so any solution to \eqref{eq:MAS} is confined to a translate of $S$. By the above remark, any solution to \eqref{eq:MAS} where $V \subset \zzp^n$ with initial condition $\xx_0 \in \rrpp^n$ is confined to $(\xx_0+S)\cap \rrpp^n$, which is called the \df{invariant polyhedron of $\xx_0$}.

\begin{figure}[h!tbp]
	\centering
	\begin{subfigure}[b]{0.3\textwidth}
	\centering 
		\begin{tikzpicture}
		\draw [opacity=0, -{stealth}, thick, blue, transform canvas={ yshift=-0.31ex}] (0,0)--(2,0) node [midway, below] {\footnotesize $1$}; 
        \draw [step=1, gray!50!white, thin] (0,0) grid (2.5,2.5);
        \node at (0,2.75) {};
		\node at (0,-0.25) {};
            \draw [->, gray] (0,0)--(2.5,0);
            \draw [->, gray] (0,0)--(0,2.5);
            \node [inner sep=0pt, blue] (0) at (0,0) {$\bullet$};
            \node [inner sep=0pt, blue] (2x2y) at (2,2) {$\bullet$};
            \draw [-{stealth}, thick, blue, transform canvas={xshift=-0.2ex, yshift=0.2ex}] (0)--(2x2y) node [near start, above] {\ratecnst{$3$}};
            \draw [-{stealth}, thick, blue, transform canvas={xshift=0.2ex, yshift=-0.2ex}] (2x2y)--(0) node [near start, below] {\ratecnst{$2$}};
            \node [inner sep=0pt, blue] (xy) at (2,0) {$\bullet$};
            \node [inner sep=0pt, blue] (2y) at (0,2) {$\bullet$};
            \draw [-{stealth}, thick, blue, transform canvas={xshift=0.2ex, yshift=0.2ex}] (2y)--(xy) node [near start, above] {\ratecnst{$3$}};
            \draw [-{stealth}, thick, blue, transform canvas={xshift=-0.2ex, yshift=-0.2ex}] (xy)--(2y) node [near start, below] {\ratecnst{$5$}};
		\end{tikzpicture}
		\caption{}
		\label{fig:intro-ex-a}
	\end{subfigure}
		\hspace{0cm}
	\begin{subfigure}[b]{0.3\textwidth}
	\centering
		\begin{tikzpicture}
        \draw [step=1, gray!50!white, thin] (0,0) grid (2.5,2.5);
        \node at (0,2.75) {};
		\node at (0,-0.25) {};
            \draw [->, gray] (0,0)--(2.5,0);
            \draw [->, gray] (0,0)--(0,2.5);
            \node [inner sep=0pt, blue] (2x) at (2,0) {$\bullet$};
            \node [inner sep=0pt, blue] (2y) at (0,2) {$\bullet$};
            \node [inner sep=0pt, blue] (0) at (0,0) {$\bullet$};
            \node [inner sep=0pt, blue] (2x2y) at (2,2) {$\bullet$};
            \draw [-{stealth}, thick, blue, transform canvas={ yshift=0.31ex}] (0)--(2x) node [midway, above] {\ratecnst{$3$}};
            \draw [-{stealth}, thick, blue, transform canvas={ yshift=-0.31ex}] (2x)--(0) node [midway, below] {\ratecnst{$5$}}; 
            \draw [-{stealth}, thick, blue, transform canvas={ yshift=0.31ex}] (2y)--(2x2y) node [midway, above] {\ratecnst{$3$}};
            \draw [-{stealth}, thick, blue, transform canvas={ yshift=-0.31ex}] (2x2y)--(2y) node [midway, below] {\ratecnst{$1$}}; 
            \draw [-{stealth}, thick, blue, transform canvas={ xshift=-0.31ex}] (0)--(2y) node [midway, left] {\ratecnst{$3$}};
            \draw [-{stealth}, thick, blue, transform canvas={ xshift=0.31ex}] (2y)--(0) node [midway, right] {\ratecnst{$3$}}; 
            \draw [-{stealth}, thick, blue, transform canvas={ xshift=-0.31ex}] (2x)--(2x2y) node [midway, left] {\ratecnst{$5$}};
            \draw [-{stealth}, thick, blue, transform canvas={ xshift=0.31ex}] (2x2y)--(2x) node [midway, right] {\ratecnst{$1$}}; 
            
            \node[outer sep=2.25pt] (diagstart) at (2,2) {};
            \node[outer sep=2.25pt] (diagend) at (0,0) {};
            \draw [-{stealth}, thick, blue] (diagstart)--(diagend) node [midway, above left] {\ratecnst{$1$}\!\!};
		\end{tikzpicture}
		\caption{}
		\label{fig:intro-ex-b}
	\end{subfigure}	
		\hspace{0cm}
	\begin{subfigure}[b]{0.3\textwidth}
	\centering
		\begin{tikzpicture}
        \draw [step=1, gray!50!white, thin] (0,0) grid (2.5,2.5);
        \node at (0,2.75) {};
		\node at (0,-0.25) {};
            \draw [->, gray] (0,0)--(2.5,0);
            \draw [->, gray] (0,0)--(0,2.5);
            \node [inner sep=0pt, blue] (2x) at (2,0) {$\bullet$};
            \node [inner sep=0pt, blue] (2y) at (0,2) {$\bullet$};
            \node [inner sep=0pt, blue] (0) at (0,0) {$\bullet$};
            \node [inner sep=0pt, blue] (2x2y) at (2,2) {$\bullet$};
            \node [inner sep=0pt, blue] (xy) at (1,1) {$\bullet$};
            \draw [-{stealth}, thick, blue] (0)--(xy) node [midway, left] {\ratecnst{$6$}};
            \draw [-{stealth}, thick, blue] (2x)--(xy) node [midway, right] {\ratecnst{$10$}}; 
            \draw [-{stealth}, thick, blue] (2y)--(xy) node [midway, left] {\ratecnst{$6$}};
            \draw [-{stealth}, thick, blue] (2x2y)--(xy) node [midway, right] {\ratecnst{$4$}}; 
		\end{tikzpicture}
		\caption{}
		\label{fig:intro-ex-c}
	\end{subfigure}	
	\caption{Weighted E-graphs from \Cref{ex:intro}.}
	\label{fig:intro-ex}
\end{figure}

\begin{ex}
\label{ex:intro} 
We illustrate the notions and notations defined above. \Cref{fig:intro-ex} shows three examples of weighted E-graphs. The graphs in \Cref{fig:intro-ex-a,fig:intro-ex-b} are weakly reversible, but that in \Cref{fig:intro-ex-c} is not. The graph in \Cref{fig:intro-ex-a} has two connected components, each of which is affinely independent; however, those in \Cref{fig:intro-ex-b,fig:intro-ex-c} do \emph{not} have affinely independent connected components.
    
The associated dynamical system of \Cref{fig:intro-ex-a} is 
\eqn{\label{eq:intro-ex} 
    \frac{d}{dt} \begin{pmatrix} x_1 \\ x_2 \end{pmatrix} 
    &= 3 \begin{pmatrix} 2 \\ 2 \end{pmatrix} 
    + 5x_1^2 \begin{pmatrix*}[r] -2 \\ 2 \end{pmatrix*} 
    + 2x_1^2x_2^2 \begin{pmatrix} -2 \\ -2 \end{pmatrix} 
    + 3x_2^2 \begin{pmatrix*}[r] 2 \\ -2 \end{pmatrix*} 
    \nonumber 
    \\&= \begin{pmatrix} 
    6 -10x_1^2 -4x_1^2x_2^2 + 6x_2^2 
    \\[5pt] 
    6 + 10x_1^2 -4x_1^2x_2^2-6x_2^2 \end{pmatrix}. 
}
The source vertices play the role of exponents in the monomials, thus the set of source vertices $V_s$ determines the monomials in the associated dynamical system. 
    
It so happens that the weighted E-graphs in \Cref{fig:intro-ex-b,fig:intro-ex-c} also have \eqref{eq:intro-ex} as their associated dynamical systems. We say that the three weighted E-graphs in \Cref{fig:intro-ex} are \emph{dynamically equivalent}, and the weighted graphs are \emph{realizations} of the dynamical system \eqref{eq:intro-ex}; we define these terms precisely in \Cref{def:DE}. This example demonstrates that while a weighted E-graph is associated to a unique dynamical system, the converse is not true; there is in general infinitely many realizations of a given polynomial dynamical system~\cite{CraciunPantea2008}. This work is concerned with finding a realization that guarantee certain algebraic and stability properties. 
\end{ex}

Another way to study the vector field generated by \eqref{eq:MAS} is to use a linear combination of some fixed vectors, one for each monomial, with the coefficients given by the strength of the monomials at that point. We give a name to those fixed vectors.

\begin{defn}
\label{def:directvect}
    Let $(V,E,\vv\kk)$ be a weighted E-graph, and $\yy_i \in V_s$. The \df{net direction vector from $\yy_i$} is 
    \eq{
        \vv w_i = \sum_{\yy_j \in V} \kk_{ij} (\yy_j - \yy_i).
    }
    The \df{matrix of net direction vectors} of $(V,E,\vv\kk)$ is 
    \eq{ 
        \mm W = \begin{pmatrix}\mtxphantom
        \vv w_1 & \vv w_2 & \cdots & \vv w_m 
        \end{pmatrix}. 
    }
\end{defn}
For convenience, we may refer to the net direction vector even if $\yy_i \not\in V_s$; in this case, let the net direction vector be zero. Such a net direction vector will \emph{not} show up as a column of $\mm W$.

The matrix $\mm W$ from \Cref{def:directvect} is also well defined when we start not with a weighted E-graph, but with a fixed polynomial dynamical system of the form 
\eqn{\label{eq:polyRHS} 
        \frac{d\xx}{dt}  = \sum_{i=1}^m \xx^{\yy_i} \vv w_i. 
    }
Note that any polynomial dynamical systems can be uniquely written as such, for some $\yy_1$, $\yy_2,\ldots, \yy_m \in \zzp^n$ distinct, and $\vv w_1$, $\vv w_2,\ldots, \vv w_m \in \rr^n$ non-zero. 

\begin{defn}
\label{def:of-ODE}
    Consider the polynomial dynamical system  \eqref{eq:polyRHS}. The \df{matrix of source vertices $\mm Y_s$} and the \df{matrix of net direction vectors $\mm W$} of  \eqref{eq:polyRHS} are 
    \begin{equation*}
        \mm Y_s = \begin{pmatrix}\mtxphantom
        \yy_1 & \yy_2 & \cdots & \yy_m 
        \end{pmatrix}
        \quad \text{and} \quad 
        \mm W = \begin{pmatrix}\mtxphantom
        \vv w_1 & \vv w_2 & \cdots & \vv w_m 
        \end{pmatrix}. 
    \end{equation*}
\end{defn}

Clearly, $\dot{\xx} = \mm W \xx^{\mm Y_s}$. 

Thus far, we start with a weighted E-graph $(V,E,\vv\kk)$, and from it, define a dynamical system. The goal of the present work is the converse direction: start with a polynomial dynamical system, find some $(V,E,\vv\kk)$, ideally with certain properties, that gives rise to such dynamics. For example, \eqref{eq:polyRHS} is generated by the graph $\yy_i \xrightarrow{\,\,1\,\,} \yy_i + \vv w_i$, for $i=1,2,\ldots, m$.  As \Cref{ex:intro} illustrates, there are in general many weighted E-graphs that can generate the same dynamics.

\begin{defn}
\label{def:DE}
    A \df{realization} of a polynomial dynamical system $\dot{\xx} = \vv f(\xx)$ is a weighted E-graph $(V,E,\vv\kk)$ whose associated dynamical system is precisely $\dot{\xx} = \vv f(\xx)$. Two realizations of $\dot{\xx} = \vv f(\xx)$ are said to be \df{dynamically equivalent}. 
\end{defn}

\begin{lem}[\cite{CraciunJinYu2019}*{Definition 2.6}] 
    The weighted E-graphs $(V,E,\vv\kk)$ and $(V',E',\vv\kk')$ are dynamically equivalent if and only if the net direction vector from $\yy_i$ in $(V,E,\vv\kk)$ coincides with that in $(V',E',\vv\kk')$, for all $\yy_i \in V_s \cup V'_s$. 
\end{lem}
\begin{proof}
    This follows from the linear independence of monomials as functions on $\rrpp^n$. 
\end{proof}

\subsection{Complex-balanced systems and \WRz\ systems}
\label{sec:intro-CB}

General polynomial dynamical systems can display a wide range of dynamical behaviours, ranging from stable or unstable steady states, limit cycles, and even chaos. In this work, we are interested in the family of \emph{complex-balanced systems}, which  enjoy various algebraic and stability properties.

\begin{defn}
\label{def:CB}
    Let $(V,E,\vv\kk)$ be a weighted E-graph in $\rr^n$, and let $\dot{\xx} = \vv f(\xx)$ be its associated dynamical system. A state $\xx^* \in \rrpp^n$ is said to be a \df{positive steady state} if $\vv f(\xx^*) = \vv 0$. Let $V_>(\vv f)$ be the set of positive steady states. A state $\xx^* > \vv 0$ is a \df{complex-balanced steady state} if at every $\yy_i \in V$, we have 
    \eq{ 
        \sum_{(i,j) \in E} \kk_{ij} (\xx^*)^{\yy_i} 
        = \sum_{(j,i) \in E} \kk_{ji} (\xx^*)^{\yy_j} . 
    }
\end{defn}
The equations above can be interpreted as balancing the fluxes flowing across the vertex $\yy_i$. If a weighted E-graph $(V,E,\vv\kk)$ admits one complex-balanced steady state, then every positive steady state is complex-balanced~\cite{HornJackson1972}; such a $(V,E,\vv\kk)$ is called a \df{complex-balanced system}. 

These systems first arose from the study of chemical systems under mass-action kinetics, as a generalization of thermodynamic equilibrium. The following theorem lists some of the most important results about complex-balanced systems. For more details, see \cite{YuCraciun2018_review, GunaNts, Feinberg1987}.

\begin{thm}[\cite{HornJackson1972}]
\label{thm:HJ}
Let $(V,E,\vv\kk)$ be a complex-balanced system, with steady state $\xx^* \in \rrpp^n$, and associated linear space $S$. Then the following are true:
\begin{enumerate}[label={(\roman*)}]
\item All positive steady states are complex-balanced, and there is exactly one steady state within each invariant polyhedron. 

\item Any complex-balanced steady state $\xx$ satisfies $\ln \xx - \ln \xx^* \in S^\perp$.

\item The function
    \eq{ 
        L(\xx) = \sum_{i=1}^n x_i(\ln x_i - \ln x^*_i - 1), 
    }
defined on $\rrpp^n$, is a strict Lyapunov function within each invariant polyhedron $(\xx_0+S)\cap \rrpp^n$, with a global minimum at the corresponding complex-balanced steady state. 

\item Every complex-balanced steady state is asymptotically stable with respect to its invariant polyhedron. 
\end{enumerate}
\end{thm}

Beside these properties, complex-balanced systems enjoy other remarkable  algebraic and dynamical properties. For example, the set of positive steady states $V_>(\vv f)$ admits a monomial parametrization~\cite{CraciunDickensteinShiuSturmfels2009, PerezmillanDickensteinShiuConradi2012}. Each positive steady state $\xx^*$ is in fact linearly stable with respect to its invariant polyhedron~\cite{Johnston_note, BorosMuellerRegensburger2020}. Complex-balanced systems are also conjectured to be \emph{persistent} and \emph{permanent}\footnote{%
    Roughly speaking, persistence is the property that starting in $\rrpp^n$, the solution is always bounded away from the boundary of $\rrpp^n$, and permanence occurs when solutions always converge to a compact subset of the invariant polyhedron.  
}~\cite{CraciunNazarovPantea2013_GAC}. Moreover, the unique steady state is conjectured to be \emph{globally} stable within its invariant polyhedron~\cite{Horn1974_GAC}. The \emph{Persistence} and \emph{Permanence Conjectures} have been proved in several cases, such as when there is only one connected component~\cite{Anderson2011_GAC, BorosHofbauer2019_GAC}, or the ambient state space is $\rr^2$~\cite{CraciunNazarovPantea2013_GAC}, or the E-graph is {\em strongly endotactic}~\cite{GopalkrishnanMillerShiu2014}, or the associated linear space $S$ is of dimension two and all trajectories are bounded~\cite{Pantea2012_GAC}. The \emph{Global Attractor Conjecture} has also been proved if    there is only one connected component~\cite{Anderson2011_GAC, BorosHofbauer2019_GAC}, or the E-graph is {strongly endotactic}~\cite{GopalkrishnanMillerShiu2014}, or the ambient state space is $\rr^3$~\cite{CraciunNazarovPantea2013_GAC}, or when the associated linear space $S$ is of dimension at most three~\cite{Pantea2012_GAC}.

Besides dynamical stability, complex-balanced systems are characterized graph-theoretically and algebraically. Horn proved in \cite{Horn1972} that $(V,E,\vv\kk)$ is complex-balanced if and only if $(V,E)$ is weakly reversible and $\vv\kk$ satisfies some algebraic equations, the number of which is measured by a non-negative integer called the \emph{deficiency} of $(V,E)$. 

\begin{defn}
\label{def:def}
    Let $(V,E)$ be an E-graph with $\ell$ connected components, and let $S$ be its associated linear space. The \df{deficiency} of $(V,E)$ is the integer $\delta = |V| - \ell - \dim S$.
\end{defn}

The notion of deficiency can also be applied  to the connected components. Suppose $V_1$, $V_2,\ldots, V_\ell$ are the connected components of $(V,E)$. The \df{deficiency of a connected component $V_p$} is $\delta_p = |V_p| - 1 - \dim S(V_p)$. It is easy to see that 
    \eq{ 
        \delta \geq \sum_{p=1}^\ell \delta_p,
    }
with equality if and only if  $S(V_1)$, $S(V_2),\ldots, S(V_\ell)$ are linearly independent. If $\delta = 0$, then necessarily $\delta_p = 0$ for all $p$.

If $(V,E)$ is weakly reversible and $\delta = 0$, then the associated dynamical system is always complex-balanced, regardless of the choice of $\vv\kk$. This result is known as the Deficiency Zero Theorem~\cite{Feinberg1972,horn1972necessary}. The deficiency is a property of the E-graph, not of the associated dynamical system, yet in the case of deficiency zero, it has strong implications on the dynamics. The goal of this paper is to search for weakly reversible and deficiency zero (\WRz) realizations for polynomial dynamical systems, which are automatically complex-balanced, and therefore obey the properties listed in \Cref{thm:HJ}.

Deficiency also has a \emph{geometric interpretation}; $\delta = 0$ if and only if $(V,E)$ has affinely independent connected components $S_1$, $S_2,\ldots, S_\ell$, and the subspaces $S(V_1)$, $S(V_2), \ldots, S(V_\ell)$ are linearly independent~\cite[Theorem 9]{CraciunJohnstonSzederkenyiTonelloTothYu2020}. Later we make use of this interpretation when searching for \WRz\ realizations.

The system \eqref{eq:MAS} admits a matrix decomposition that aids in studying complex-balanced steady states. For a weighted E-graph $(V,E,\vv\kk)$ where $|V| = m$, its associated dynamical system \eqref{eq:MAS} can decomposed be as $\dot{\xx} = \mm Y \mm A_{\vv\kk} \xx^{\mm Y}$~\cite{HornJackson1972}, where 
    \eq{ 
        \mm Y = \begin{pmatrix} \mtxphantom
        \yy_1 & \yy_2 & \cdots & \yy_m 
        \end{pmatrix} 
    }
is a matrix whose columns are the vertices (including both sources and targets); $\xx^{\mm Y}$ is the vector of monomials whose $i$th component is $\xx^{\yy_i}$, and the \df{Kirchoff matrix} 
    \eq{ 
        [\mm A_{\vv\kk}]_{ij} = \left\{ 
        \begin{array}{cl}
            \kk_{ji} &  \text{if } \yy_i \to \yy_j \in E
            \\[5pt] \ds -\sum_r \kk_{jr} & \text{for } i = j
            \\[5pt] 0 & \text{otherwise} 
        \end{array} 
        \right. 
    }
is the negative transpose of the graph Laplacian of $(V,E,\vv\kk)$. In general, the $i$th component of $\mm A_{\vv\kk}\xx^{\mm Y}$ 
\eq{ 
    [\mm A_{\vv\kk} \xx^{\mm Y}]_i = \sum_{(j,i) \in E} \kk_{ji} \xx^{\yy_j} - \xx^{\yy_i} \!\! \sum_{(i,j) \in E} \kk_{ij}
}
measures the net flux passing through the $i$th vertex, so a complex-balanced steady state $\xx^*$ is a solution to the equation $\mm A_{\vv\kk} (\xx^*)^{\mm Y} = \vv 0$.

A subgraph $(V_0, E_0) \subseteq (V,E)$ is a \df{terminal strongly connected component} if it is strongly connected, and there does not exist an edge in $E$ from a vertex in $V_0$ to a vertex in $V \setminus V_0$. The kernel of $\mm A_{\vv\kk}$ is supported on the terminal strongly connected components: 

\begin{thm}[\cite{FeinbergHorn1977}]
\label{thm:kerAk}
Let $\mm A_{\vv\kk}$ be the Kirchoff matrix of $(V,E,\vv\kk)$ with terminal strongly connected components $V_1$, $V_2,\ldots, V_t$. There exists a basis $\{\vv c_1, \vv c_2,\ldots, \vv c_t \}$ for $\ker \mm A_{\vv\kk}$ with $\vv c_p \in \rrp^n$ and
\eq{ 
    \left\{ 
    \begin{array}{cl}
         [\vv c_p]_i  > 0 & \text{ if } \yy_i \in V_p, \\[5pt]
         [\vv c_p]_i  = 0 & \text{ otherwise.}
    \end{array}
    \right. 
}
\end{thm}

According to the Matrix-Tree Theorem~\cite{CraciunDickensteinShiuSturmfels2009, horn1972necessary}, there is an explicit formula for the entries of $\vv c_p$. Each non-zero $[\vv c_p]_i$ is a polynomial of $\kk_{ij}$ with positive coefficients, given by the maximal minors of $\mm A_{\vv\kk}$~\cite{CraciunDickensteinShiuSturmfels2009, StanleyFomin1999_MtxTree, MirzaevGunawardena2013}.

If $(V,E)$ is weakly reversible, then $\delta =  \dim(\ker \mm Y \cap \ran \mm A_{\vv\kk})$. More generally, $\dim(\ker \mm Y \cap \ran \mm A_{\vv\kk}) = |V| - \ell - t$, when $(V,E,\vv\kk)$ has $t$ terminal strongly connected components~\cite{FeinbergHorn1977}. 
Therefore if $(V,E)$ is \WRz, then $\ker (\mm Y \mm A_{\vv\kk}) = \ker \mm A_{\vv\kk}$, and the matrix of net direction vectors $\mm W$ coincides with $\mm Y \mm A_{\vv\kk}$ (see \Cref{lem:mtx}). 

For the purpose of this work, we assume that we are given $\mm W$ and the matrix of source vertices $\mm Y_s$, but we do not know the decomposition of $\mm W$ into the product $\mm Y \mm A_{\vv\kk}$, where the columns of $\mm Y_s$ are also columns of $\mm Y$. Because $\ker \mm A_{\vv\kk}$ is well characterized~\cite{FeinbergHorn1977, GunaNts, feinberg2019foundations}, we make use of it in our search for \WRz\ realizations.

\begin{figure}[t!]
	\centering
	\begin{tikzpicture}
        \draw [step=1, gray!50!white, thin] (0,0) grid (4.75, 3.5);
        \node at (0,3.5) {};
		\node at (0,-0.25) {};
            \draw [->, gray] (0,0)--(4.75,0);
            \draw [->, gray] (0,0)--(0,3.5); 
        \draw [rounded corners, very thick, viridisturq, fill = viridisyellow, fill opacity=0.2] (-0.75,1.75) rectangle (2.75,3.25); 
        \draw [rounded corners, very thick, viridisturq, fill = viridisyellow, fill opacity=0.2] (-0.6,-0.5) rectangle (2.3,1.5);
        \draw [rounded corners, very thick, viridisturq, fill = viridisyellow, fill opacity=0.2] (3.4,-0.5) rectangle (4.6,2.5); 
            \node [inner sep=0pt, outer sep=1.2pt, blue] (00) at (0,0) {$\bullet$};
            \node [inner sep=0pt, blue] (x) at (1,0) {$\bullet$};
            \node [inner sep=0pt, blue] (2xy) at (2,1) {$\bullet$};
            \node [inner sep=0pt, blue] (4x2y) at (4,2) {$\bullet$};
            \node [inner sep=0pt, blue] (y) at (0,2) {$\bullet$};
            \node [inner sep=0pt, blue] (x2y) at (2,3) {$\bullet$};
            \node [inner sep=0pt, blue] (3xy) at (3,1) {$\bullet$};
            \node [inner sep=0pt, blue] (4x) at (4,0) {$\bullet$};
            \draw [-{stealth}, thick, blue, transform canvas={xshift=-0.05ex, yshift=0.25ex}] (y)--(x2y) node [midway, above] {\ratecnst{$\kk_{12}$}};
            \draw [-{stealth}, thick, blue, transform canvas={xshift=0.05ex, yshift=-0.25ex}] (x2y)--(y) node [midway, below] {\ratecnst{$\kk_{21}$}};
            \draw [-{stealth}, thick, blue, transform canvas={xshift=-0.25ex, yshift=0ex}] (4x)--(4x2y) node [midway, left] {\ratecnst{$\kk_{34}$\!\!}};
            \draw [-{stealth}, thick, blue, transform canvas={xshift=0.25ex, yshift=0ex}] (4x2y)--(4x) node [midway, right] {\ratecnst{\!\!$\kk_{43}$}};
            \draw [-{stealth}, thick, blue, transform canvas={xshift=0ex, yshift=0ex}] (00)--(x) node [midway, below] {\ratecnst{$\kk_{56}$}};
            \draw [-{stealth}, thick, blue, transform canvas={xshift=0ex, yshift=0ex}] (x)--(2xy) node [midway, below right] {\ratecnst{\!\!$\kk_{67}$}};
            \draw [-{stealth}, thick, blue, transform canvas={xshift=0ex, yshift=0ex}] (2xy)--(00) node [midway, above] {\ratecnst{$\kk_{75}$}};
            \draw [-{stealth}, thick, blue, transform canvas={xshift=0ex, yshift=0ex}] (3xy)--(2xy) node [midway, above] {\ratecnst{\,\,\,\,$\kk_{87}$}};
            \draw [-{stealth}, thick, blue, transform canvas={xshift=0ex, yshift=0ex}] (3xy)--(4x2y) node [midway, left] {\ratecnst{$\kk_{84}$}};
            \node at (y) [left] {\footnotesize $\yy_1$};
            \node at (x2y) [right] {\footnotesize $\yy_2$};
            \node at (4x) [below] {\footnotesize $\yy_3$};
            \node at (4x2y) [above] {\footnotesize $\yy_4$};
            \node at (00) [below left] {\footnotesize $\yy_5$};
            \node at (x) [below right] {\footnotesize $\yy_6$};
            \node at (2xy) [above] {\footnotesize $\yy_7$};
            \node at (3xy) [below] {\footnotesize $\yy_8$};    
	\end{tikzpicture}
	\caption{A weighted E-graph with two connected components but three terminal strongly connected components (boxed). Its Kirchoff matrix $\mm A_{\vv\kk}$ and a basis for $\ker \mm A_{\vv\kk}$ are given in \Cref{ex:kerAk}.}
	\label{fig:ex-kerA}
\end{figure}

\begin{ex}
\label{ex:kerAk}
Consider the weighted E-graph $(V,E,\vv\kk)$ in \Cref{fig:ex-kerA}. While $(V,E)$ has two connected components, it has three terminal strongly connected components (boxed in \Cref{fig:ex-kerA}). With the ordering of vertices as labelled in the figure, the Kirchoff matrix of $(V,E,\vv\kk)$ is 
    \eq{ 
        \mm A_{\vv\kk} = 
        \begin{pmatrix}
            - \kk_{12} & \hphantom{-}\kk_{21} \\
            \hphantom{-}\kk_{12} & -\kk_{21} \\ 
            && -\kk_{34} & \hphantom{-}\kk_{43} \\
            && \hphantom{-}\kk_{34} & -\kk_{43} &&&& \kk_{84} \\
            &&&& -\kk_{56} & \hphantom{-}0 & \hphantom{-}\kk_{75} \\
            &&&& \hphantom{-}\kk_{56} & -\kk_{67} & \hphantom{-}0 \\
            &&&& \hphantom{-}0 & \hphantom{-}\kk_{67} & -\kk_{75} & \kk_{87}  \\
            &&&&&&& -\kk_{84} -\kk_{87} 
        \end{pmatrix}. 
    }
A basis for its kernel is given by the vectors 
    \eq{ 
        \vv c_1 = \begin{pmatrix} \kk_{21} \\ \kk_{12}  \\ 0\\ 0 \\ 0\\0\\0\\0 \end{pmatrix} ,
        \qquad 
        \vv c_2 = \begin{pmatrix} 0 \\ 0  \\ \kk_{43}\\ \kk_{34} \\ 0\\0\\0\\0 \end{pmatrix} ,
        \qquad 
        \vv c_3 = \begin{pmatrix} 0\\0 \\ 0\\ 0 \\ \kk_{67}\kk_{75} \\ \kk_{56}\kk_{75} \\ \kk_{56}\kk_{67}\\0 \end{pmatrix} .
    }
The supports of the basis vectors $\vv c_p$ are precisely the terminal strongly connected components of $(V,E)$. If the graph is weakly reversible, then the basis of $\ker \mm A_{\vv\kk}$ given in \Cref{thm:kerAk} provides a way to partition the set of vertices. 
\end{ex}

\section{Main results}
\label{sec:main-result}

In this section, we present \Cref{algorithm} (see also \Cref{fig:WRzAlg}) that searches for a weakly reversible and deficiency zero (\WRz) realization of a given system of polynomial differential equations 
\eqn{\label{eq:alg-intro} 
    \frac{d\xx}{dt} = \sum_{i=1}^m \xx^{\yy_i} \vv w_i, 
}
where $\yy_1,\ldots, \yy_m \in \zzp^n$ are distinct, and $\vv w_1$, $\vv w_2,\ldots, \vv w_m \in \rr^n\setminus\{\vv 0\}$. Whenever \eqref{eq:alg-intro} admits a \WRz\ realization, the system is complex-balanced and enjoys all the properties listed in \Cref{thm:HJ}.  Moreover, if no \WRz\ realization exists for \eqref{eq:alg-intro}, our algorithm would conclude as much. Whenever a \WRz\ realization exists, the set of positive steady states has a log-linear structure that allows us to easily find the steady states of \eqref{eq:alg-intro}, as outlined in \Cref{thm:eqm}. Finally, our algorithm is valid even if the $\vv w_i$'s are only known up to a positive scalar multiple; we prove this in \Cref{thm:vark}.

\subsection{Algorithm for \WRz\ realization}
\label{sec:alg-pf}

\tikzset{%
    Node/.style={rectangle, rounded corners, draw=black, thick, fill=blue!10, fill opacity = 1, minimum width=4.5cm, minimum height=1cm, outer sep=0pt},
    Edge/.style={very thick, style={very thick, arrows={-Stealth[length=7.5pt,width=7.5pt]}}},
} 
\begin{figure}[p]
\centering
    \begin{tikzpicture}
    \node[Node, fill=yellow!70, minimum width=5cm, minimum height=1.5cm] (input) at (2.2,-0.25) {}; 
        \node at (0.1,0) [right] {\small $\mm Y_s = \begin{pmatrix} \yy_1 ,  \ldots ,  \yy_m\end{pmatrix} \in  \zzp^{n\times m}$};
        \node at (0,-0.5) [right] {\small $\mm W = \begin{pmatrix} \vv w_1 ,  \ldots ,  \vv w_m\end{pmatrix} \in \rr^{n\times m}$};
    \node[Node, minimum height=2cm, minimum width=5.75cm] (genCone) at (9,-0.25) {};
        \node at (9,0.25) {\small Find minimal set of generators};
        \node at (9,-0.25) {\small $\{\vv c_1,\ldots, \vv c_\ell\} \subset \rrp^m$};
        \node at (9, -0.75) {\small for the cone $\ker \mm W \cap \rrp^m$};
    \draw[Edge] (input)--(genCone) ;
    \node[Node,  minimum width=5.75cm, minimum height=1.5cm] (partition) at (9,-2.75) {};
        \node at (9,-2.5) {\small Does $\{ \mrm{supp}(\vv c_p)\}_{p=1}^\ell$};
        \node at (9,-3) {\small partition $[1,m] \subseteq \nn$?};
    \draw[Edge] (genCone)--(partition);
    \node[Node,  minimum width=5.75cm, fill=orange!30] (fail1) at (2.2,-2.75) {\small No \WRz\ realization};
    \draw[Edge] (partition)node[xshift=-3.3cm, above] {\small no}--(fail1);
    \node[Node, minimum width=5.75cm, minimum height=1.5cm] (linkage) at (9,-5.25) {};
        \node at (9,-5) {\small Define candidate connected};
        \node at (9,-5.5) {\small components by $V_p \coloneqq \mrm{supp}(\vv c_p)$};
    \draw[Edge] (partition)--(linkage) node[midway, right] {\small yes};
    \node[Node,  minimum width=5.75cm, minimum height=1.5cm] (aff) at (9,-8.25) {};
        \node at (9,-8) {\small Is $\{\yy_i \colon i \in V_p\}$};
        \node at (9,-8.5) {\small affinely independent?};
    \draw[Edge] (linkage)--(aff) node [midway, left] {\small \ttfamily\bfseries \blue{for} p = 1,2,...,l \,};
    \draw[Edge] (aff)node[xshift=-3.3cm, above] {\small no}--(3.5,-8.25) --(3.5, -3.25)  ;
    \node[Node, minimum height=1.5cm, minimum width=5.75cm] (lindecomp) at (9,-10.75) {};
        \node at (9,-10.5) {\small $\forall \, i \in V_p$, is };
        \node at (9,-11) {\small $\vv w_i \in \mrm{Cone}\{\yy_j - \yy_i \colon j \in V_p\}$?};
    \draw[Edge] (aff)--(lindecomp) node [midway, right] {\small yes}; 
    
    \node[Node, minimum height=2.1cm, minimum width=5.75cm] (WR) at (9,-13.5) {};
        \node at (9,-13) {\small Uniquely decompose each $\vv w_i$};
        \node at (9,-13.1) [below] {\small as $\ds \vv w_i = \sum_{\substack{j\neq i\\j \in V_p}} \kk_{ij} (\yy_j - \yy_i) $};
    \draw[Edge] (lindecomp)--(WR) node [right, midway] {\small yes};
    \draw[Edge] (lindecomp)node[xshift=-3.3cm, above] {\small no}--(2.2,-10.75) --(2.2, -3.25)  ;
    \draw[Edge] (WR) to node[right] {\small } (9,-15.25) 
        to (13,-15.25)
        to 
        (13,-6.75)--(9, -6.75); 
    \node[Node,  minimum width=5.75cm, fill=green!30] (success) at (9, -16.75) {\small \WRz\ realization found};
    \draw[Edge] (9,-15.25)--(success) ; 
    \node at (9,-15.25) [left] {\small \ttfamily\bfseries \blue{endfor}\,};
\end{tikzpicture}  
\caption{\Cref{algorithm} for finding \WRz\ realization of a polynomial dynamical system $\dot{\xx} = \sum_{i=1}^m \xx^{\yy_i} \vv w_i$. 
}
\label{fig:WRzAlg} 
\end{figure}

The inputs of \Cref{algorithm} are the source vertices and their net direction vectors via $\mm Y_s$ and $\mm W$ respectively. To find a \WRz\ realization $(V,E,\vv\kk)$ is to find a matrix decomposition of $\mm W = \mm Y_s \mm A_{\vv\kk}$, where $\mm A_{\vv\kk}$ encodes the graph structure of $(V,E)$. In the following lemma, we prove properties that can be expected should a \WRz\ realization exists. 

Recall that a set $X$ is a \df{polyhedral cone} if $X = \{ \vv x \st \mm M \vv x \leq \vv 0\}$ for some matrix $\mm M$. Such a cone is convex. It is \df{pointed}, or \df{strongly convex}, if it does not contain a positive dimensional linear subspace. Note that a cone contained in the positive orthant $\rrp^m$ is always pointed. A pointed polyhedral cone admits a unique (up to scalar multiple) minimal set of generators~\cite{CoxLittleSchenck2011}.

\begin{algorithm}[H]
\caption{(\WRz\ algorithm)}
\label{algorithm}
\begin{algorithmic}[1]
\REQUIRE Matrices  $\mm Y_s = \begin{pmatrix}  \yy_1,\ldots, \yy_m \end{pmatrix} $ and  $\mm W = \begin{pmatrix} \vv w_1,\ldots, \vv w_m \end{pmatrix}$ that define  $\dot{\xx} = \sum_{i=1}^m \xx^{\yy_i} \vv w_i$. 

\phantom{ }

\ENSURE Either return the unique \WRz\ realization, or print that such a realization does not exist.

\phantom{} 

\STATE Find the minimal set of generators
        $\{\vv c_1, \vv c_2, \ldots, \vv c_\ell\}$ of the pointed convex cone  $\ker \mm W \cap \rrp^m$.
        
\IF{the sets $\mrm{supp}(\vv c_1), \mrm{supp}(\vv c_2), \dots, \mrm{supp}(\vv c_\ell)$ do \underline{not} form a partition of  $\{1,2, \dots, m\}$} 
    \STATE {\bf Print:} {\tt WR0 realization does not exist.} {\bf Exit.} \label{alg:exit1}
\ELSE 
    \STATE Define the candidate connected components to be $V_p \coloneqq \mrm{supp}(\vv c_p)$, for $p = 1, 2, \dots, \ell.$
    
    \FOR{$p = 1, 2, \dots, \ell$}
    \IF{ the vectors $\{\yy_i \colon i \in V_p\}$ are \underline{not} affinely independent} \label{alg:if-affine}
        \STATE {\bf Print:} {\tt WR0 realization does not exist.} {\bf Exit.} \label{alg:exit2}
    \ELSE 
        \FOR{each $i \in V_p$}
            \IF{$\vv w_i \notin \mrm{Cone}\{\yy_j - \yy_i \colon j \in V_p\}$}  \label{alg:if-cone}
                \STATE {\bf Print:} {\tt WR0 realization does not exist.} {\bf Exit.} \label{alg:exit3}
            \ELSE
                \STATE {Uniquely decompose $\vv w_i = \sum_j \kk_{ij} (\yy_j - \yy_i)$ with $\kk_{ij} \geq 0$.} \label{alg:decompose-wi} 
                \STATE {Add $\{ \yy_i \to \yy_j \st \kk_{ij} > 0\}$ to edge set $E$.}
            \ENDIF 
        \ENDFOR 
    \ENDIF
    \ENDFOR 
\ENDIF

\STATE {\bf Print:} {\tt The WR0 realization does exist, and has $\ell$ connected components.} 

\STATE {\bf Print:} {\tt The connected components of this realization are given by $\{V_p\}_{p=1}^\ell$.}  \label{alg:WR0-V}

\STATE {\bf Print:} {\tt The edges are given by $E$, with weights $\kk_{ij}$.
    }\label{alg:WR0-E}
\end{algorithmic}
\end{algorithm}

\begin{lem}
\label{lem:mtx} 
    Suppose a polynomial dynamical system $\dot{\xx} = \vv f(\xx)$ admits a \WRz\ realization $(V,E,\vv\kk)$ with $\ell$ connected components. Let $\mm Y_s$ be the matrix of source vertices, and $\mm W$ the matrix of net direction vectors of the polynomial dynamical system $\dot{\xx} = \vv f(\xx)$. Let $S$ be the associated linear space, and $\mm A_{\vv\kk}$ the Kirchoff matrix of the weighted E-graph $(V,E,\vv\kk)$. Then we have: 
    \begin{enumerate}[label={(\roman*)}]
    \item $\mm W = \mm Y_s \mm A_{\vv\kk}$ and $\ker \mm W = \ker \mm A_{\vv\kk}$, 
    \item $S = \ran \mm W$, and the rank of $\mm W$ is $|V|-\ell$, 
    \item $\ker \mm W \cap \rrp^m$ is a pointed polyhedral cone, and  
    \item a minimal set of generators for $\ker \mm W \cap \rrp^m$ has $\ell$ elements, whose supports correspond to the connected components of $(V,E)$. \label{lem:mtx-gen}
    \end{enumerate}
\end{lem}
\begin{proof}
Because $(V,E)$ is weakly reversible, all vertices in $V$ are sources. Moreover, because the deficiency of $(V,E)$ is zero, by \cite[Proposition 3.5]{CraciunJinYu_unique}, the net direction vector from any $\yy_i$ is non-zero, and the set of source vertices corresponds exactly to the columns of $\mm Y_s$. 
\begin{enumerate}[label={(\roman*)}]
\item  By definition, $\vv f(\xx) = \mm W \xx^{\mm Y_s}$, and by dynamical equivalence, $\vv f(\xx) = \mm Y_s \mm A_{\vv\kk} \xx^{\mm Y_s}$. Since the coefficients of polynomial functions are uniquely determined, $\mm W = \mm Y_s \mm A_{\vv\kk}$. Because  $\dim (\ker \mm Y_s \cap \ran \mm A_{\vv\kk} ) = \delta = 0$, we have $\ker \mm W = \ker \mm A_{\vv\kk}$. 

\item Note that 
\eq{ 
    \rank \mm W = \rank \mm A_{\vv\kk} = |V| - \ell = \dim S, 
}
where the first and last equalities follow from $0 = \delta  = \dim (\ker \mm Y_s \cap \ran \mm A_{\vv\kk}) = |V| - \ell - \dim S$, and the second equality follows from weak reversibility and \Cref{thm:kerAk}. Clearly $\ran \mm W \subseteq S$, so  $\ran \mm W = S$.

\item  The set $\ker \mm W \cap \rrp^m$ is the solution to $\mm W \vv \nu  \geq \vv 0$, $- \mm W \vv \nu \geq \vv 0$, and $\mm{Id} \, \vv \nu \geq \vv 0$; thus the set is a polyhedral cone. That $\ker \mm W \cap \rrp^m$ is pointed follows from it being a subset of $\rrp^m$. 

\item Let $\mc B = \{\vv c_1$, $\vv c_2, \ldots, \vv c_\ell\}$ be a basis of $\ker \mm A_{\vv\kk}$ as in \Cref{thm:kerAk}, where $\vv c_p \geq \vv 0$, and each $V_p = \{ \yy_i \st i \in \mrm{supp}(\vv c_p)\}$ is a connected component of $(V,E)$. Clearly $\mc B \subseteq \ker \mm W \cap \rrp^m$; we claim that $\mc B$ is a \emph{minimal} set of \emph{generators} for the pointed cone. 

Let $\vv \nu \in \ker \mm W \cap \rrp^m$ be arbitrary. By (ii), $\mc B$ is a basis for $\ker \mm W$, so decompose $\vv \nu$ accordingly: 
    \eq{ 
        \vv \nu = \sum_{p=1}^\ell \lambda_p \vv c_p, 
    }
for some $\lambda_p \in \rr$.
By \Cref{thm:kerAk}, each $\vv c_p$ is supported on the connected components of $(V,E)$, which partition the set of vertices. In particular, for each $i = 1,\ldots, m$, there is exactly one $p(i)$ such that $\nu_i = \lambda_{p(i)} [\vv c_{p(i)}]_i$. Since $\vv \nu$, $\vv c_{p(i)} \in \rrp^m$, it must be the case that $\lambda_{p(i)} \geq 0$. In other words, $\mc B$ generates the cone $\ker \mm W \cap \rrp^m$. Because the vectors in $\mc B$ have disjoint supports, $\mc B$ is minimal.  \qedhere
\end{enumerate}
\end{proof}

\begin{lem}
\label{claim1}
    If \Cref{algorithm} exits at lines~\ref{alg:exit1}, \ref{alg:exit2}, or \ref{alg:exit3}, then $\dot{\xx} = \sum_{i=1}^m \xx^{\yy_i} \vv w_i$ does not admit a \WRz\ realization. 
\end{lem}
\begin{proof}
    If the algorithm exits at line~\ref{alg:exit1}, then by the contrapositive of \Cref{lem:mtx}\ref{lem:mtx-gen} no \WRz\ realization exists. Continuing with the algorithm, let $\{ \vv c_1,\ldots, \vv c_\ell\}$ be a minimal set of generators of $\ker \mm W \cap \rrp^n$, and partition the vertices as $V_p \coloneqq \mrm{supp}(\vv c_p)$.  
    If instead the algorithm exits at line~\ref{alg:exit2}, then again no \WRz\ realization exists because \WRz\ realizations have affinely independent connected components~\cite[Theorem 9]{CraciunJohnstonSzederkenyiTonelloTothYu2020}. 
    Finally, exiting at line~\ref{alg:exit3} means that some net direction vector $\vv w_i$ cannot be decomposed as edges from $\yy_i$ to other vertices in $V_p$, which defines a connected component of a \WRz\ realization if it exists according to \Cref{lem:mtx}\ref{lem:mtx-gen}. 
\end{proof}

\begin{lem}
\label{claim2a}
    Suppose \Cref{algorithm} reaches line~\ref{alg:WR0-E}. Then the connected components of $(V,E)$ are given by $V_1$, $V_2,\ldots, V_\ell$.
\end{lem}
\begin{proof}
    If the algorithm reaches line~\ref{alg:WR0-E}, a realization has been found with edges among $V_1,\ldots V_\ell$, i.e., the connected components are subsets of $V_p$. We prove now that in fact, each $V_p$ is connected in $(V,E)$.
    
    For any $p=1,\ldots, \ell$, let $U \coloneqq V_p$ be the support of $\vv c \coloneqq \vv c_p$. Suppose for a contradiction that $V^* \subsetneq U$ is a connected component. Because the \textbf{if} statement in line~\ref{alg:if-affine} is false, $U$ is affinely independent, so the linear subspaces 
    \eq{ 
        S(V^*) = \{ \yy_i - \yy_j \st \yy_i, \yy_j \in V^*\}
        \quad \text{and} \quad 
        S(U \setminus V^*) = \{ \yy_i - \yy_j \st \yy_i, \yy_j \in U \setminus V^*\}
    }
    are linearly independent. Because $\vv c \in \ker \mm W \cap \rrp^m$ and $\mrm{supp}(\vv c) = U$, we have 
    \eqn{\label{eq:pf-connected-component} 
        \vv 0 & 
        = \sum_{i \in V^*} c_i \vv w_i + \!\! \sum_{i \in U \setminus V^*} \!\! c_i \vv w_i, 
    }
    with $c_i > 0$. Furthermore, the \textbf{if} statement in line~\ref{alg:if-cone} returning false implies that each $\vv w_i$ in \eqref{eq:pf-connected-component} can be further decomposed as edges between vertices in $U$. For any $\yy_i$ in $V^*$, which is a connected component, the net direction vector $\vv w_i$ is a positive linear combination of edge vectors between $\yy_i$ and other vertices in $V^*$, so $\vv w_i \in S(V^*)$. In particular, $\vv c^* \coloneqq \sum_{i \in V^*} c_i \vv w_i \in S(V^*)$. Similarly, any vertices in $U \setminus V^*$ are only connected to other vertices in $U\setminus V^*$.  Linear independence of $S(V^*)$ and $S(U\setminus V^*)$ means that the vectors $\vv c^* $ and $\vv c - \vv c^*$ are linearly independent. Both $\vv c^* $ and $\vv c - \vv c^*$  lie in the cone $\ker \mm W \cap \rrp^m$, so $\{ \vv c_1,\ldots, \vv c_\ell\}$ is not a set of generators, which is a contradiction. 
\end{proof}

\begin{lem}
\label{claim2b} 
    Suppose \Cref{algorithm} reaches line~\ref{alg:WR0-E}. Then the deficiency of $(V,E)$ is zero. 
\end{lem}
\begin{proof}
    The falsity of the \textbf{if} statement in line~\ref{alg:if-affine} and \Cref{claim2a} imply that the connected components $V_1$, $V_2,\ldots, V_\ell$ are affinely independent. To prove $\delta = 0$, it remains to show that $S(V_1)$, $S(V_2),\ldots, S(V_\ell)$ are linearly independent subspaces~\cite[Theorem 9]{CraciunJohnstonSzederkenyiTonelloTothYu2020}. 
    
    We claim that the minimal set of generators $\{ \vv c_1,\vv c_2,\ldots, \vv c_\ell\}$ forms a basis for $\ker \mm W$. Let $\vv c \in \ker \mm W$ be arbitrary. If $\vv c$ has non-negative components, then it is a linear combination of $\vv c_p$'s. If $\vv c \not\in \rrp^m$, then there exist sufficiently large constants $\mu_p > 0$ so that 
    \eq{ 
         \vv c + \sum_{p=1}^\ell \mu_p \vv c_p \in \rrp^m. 
    }
    This vector is a non-negative combination of $\vv c_1$, $\vv c_2, \ldots, \vv c_\ell$; thus, $\vv c$ is a linear combination of the generating vectors. Since $\mm W \in \rr^{n \times |V|}$, we have $\rank \mm W = |V| - \ell$. 
    
    Let $S$ be the associated linear space of $(V,E)$, i.e., $S = \mrm{span} \{ \yy_j - \yy_i \st \yy_i \to \yy_j \in E\}$. The falsity of the \textbf{if} statement in line~\ref{alg:if-cone} implies that $\ran \mm W \subseteq S$, so $\dim S \geq |V| - \ell$. Hence, the deficiency of $(V,E)$ is $\delta = |V| - \ell - \dim S \leq 0$. Because $\delta$ is always non-negative, we conclude that $\delta = 0$. 
\end{proof}

\begin{lem}
\label{claim2c}
    Suppose \Cref{algorithm} reaches line~\ref{alg:WR0-E}. Then $(V,E)$ is weakly reversible. 
\end{lem}   
\begin{proof}
    Without loss of generality, we reorder the vertices according to the connected components of the deficiency zero realization $(V,E,\vv\kk)$ found in line~\ref{alg:WR0-E}. Let $\yy_1$, $\yy_2, \ldots, \yy_{m_1}$ form the first connected component; $\yy_{m_1+1}, \ldots, \yy_{m_2}$ form the second connected component, and so on. Accordingly, reorder the columns of the matrix $\mm W$ of net direction vectors, the indices of the generators $\vv c_p$, and their supports $V_p$. Under the new ordering of vertices, $\mrm{supp}(\vv c_p) = V_p = \{ \yy_{m_{p-1}+1},\ldots, \yy_{m_p}\}$ for $p=1$, $2,\ldots, \ell$, where we take $m_0 = 0$.
    
    With the new ordering of vertices, the Kirchoff matrix $\mm A_{\vv\kk}$ of $(V,E)$ is block-diagonal. Denote by $\mm A_{\vv\kk}^{(p)}$ the $p$th block, which encodes the connectivity of the connected component defined by $V_p$. Indeed, $\mm A_{\vv\kk}^{(p)}$ is the Kirchoff matrix of the $p$th connected component when viewed as a weighted E-graph in the subgraph sense. 
    
    We now prove that the first component $(V_1,E_1)$ is strongly connected. Let $\mm Y^{(1)}$ be the first $m_1$ columns of $\mm Y_s$ and $\mm W^{(1)}$ be the first $m_1$ columns of $\mm W$, so that $\mm Y^{(1)} \mm A_{\vv\kk}^{(1)} = \mm W^{(1)}$. Let $\vv c_1 = \vv c_1' \oplus \vv 0 \in \rrp^n$, with $\vv c_1' \in \rrpp^{m_1}$ spanning the one-dimensional subspace $\ker \mm W^{(1)}$. Finally, let $S_1 = \Span \{ \yy_j -\yy_i \st \yy_i \in V_1\}$. 
    
    Because $(V_1, E_1)$ is connected and $V_1$ is affinely independent, $\dim S_1 = |V_1| - 1$, so $\delta_1 = 0$. Moreover, if $t$ denotes the number of terminal strongly connected components in $(V_1,E_1)$, then $\dim (\ker \mm Y^{(1)} \cap \ran \mm A_{\vv\kk}^{(1)}) = |V_1| - t - \dim S_1$~\citelist{\cite{FeinbergHorn1977} \cite{GunaNts}*{Proposition 3.1}}. Since $|V_1|-t- \dim S_1 \leq \delta_1$, we conclude that $\ker (\mm Y^{(1)} \mm A_{\vv\kk}^{(1)}) = \ker \mm A_{\vv\kk}^{(1)}$, and $\vv c_1'$ also spans $\ker \mm A_{\vv\kk}^{(1)}$. By \Cref{thm:kerAk}, $\vv c_1'$ is supported on the terminal strongly connected component, which in this case is all of $V_1$. Therefore, $(V_1,E_1)$ is in fact strongly connected. 
    
    An analogous claim can be made about the other connected components. Consequently $(V,E)$ is weakly reversible.     
\end{proof}

\begin{rmk}
\label{rmk:stoich-linkage}
In the proof of \Cref{claim2a} and \Cref{claim2b}, we have proved that the associated linear subspace $S(V_p)$ is in fact the span of the net direction vectors belonging to said connected component. Thus, there are multiple ways of generating $S(V_p)$:
\eq{ 
    \Span\{ \yy_j - \yy_i \st \yy_i, \yy_j \in V_p \} 
    = \Span\{ \yy_j - \yy_i \st  \yy_i \to \yy_j \in E_p\} 
    =\ran \mm W^{(p)}.
}
\end{rmk}

\bigskip

The lemmas above provide the technical parts that we need to prove the main result of this paper. 
\begin{thm}
\label{thm:alg}
Given a system of differential equations
    \eq{ 
        \frac{d\xx}{dt} = \sum_{i=1}^m \xx^{\yy_i} \vv w_i,
    }
with distinct $\yy_i \in \zzp^n$, and $\vv w_i \in \rr^n\setminus\{\vv 0\}$.  \Cref{algorithm} returns the unique \WRz\ realization of the dynamical system if it exists, or concludes that no \WRz\ realization exists.
\end{thm}
\begin{proof}
    There are two possible scenarios: either (1) the algorithm exits at lines~\ref{alg:exit1}, \ref{alg:exit2}, or \ref{alg:exit3} by failing one of the \textbf{if} statements, or (2) the algorithm successfully reaches line~\ref{alg:WR0-E}. In the first scenario, \Cref{claim1} implies that no \WRz\ realization exists. In the second scenario, the realization has connected components $V_1$, $V_2,\ldots, V_\ell$ according to \Cref{claim2a}. The realization is weakly reversible and deficiency zero by \Cref{claim2b,claim2c} respectively. The uniqueness of the realization follows from \cite{CraciunJinYu_unique}.
\end{proof}

\begin{rmk}
The uniqueness of the \WRz\ realization is also a consequence of \Cref{algorithm}. This is due to the affine and linear independences, as well as the structure of $\ker \mm W = \ker \mm A_{\vv\kk}$. 
\end{rmk}

If a \WRz\ realization exists, the polynomial dynamical system is complex-balanced. Therefore, if a system passes \Cref{algorithm}, it automatically inherits all the algebraic and dynamical properties of complex-balanced system. Weak reversibility implies that a positive steady state exists~\cite{Boros2019}. The remaining statements in the theorem below are easy consequence of \Cref{thm:HJ,thm:alg}.

\begin{thm} 
\label{cor:alg}
Suppose the system of differential equations
    \eqn{\label{eq:poly-cor} 
        \frac{d\xx}{dt} = \sum_{i=1}^m \xx^{\yy_i} \vv w_i,
    }
with distinct  $\yy_i \in \zzp^n$ and $\vv w_i \in \rr^n\setminus \{\vv 0\}$, passes \Cref{algorithm}. Let $\mm W$ be the matrix of net direction vectors and $S = \ran \mm W$. Then the following holds. 
\begin{enumerate}[label={(\roman*)}]
\item A positive steady state $\xx^*$ exists. 
\item There is exactly one steady state within every invariant polyhedron $(\xx_0 + S) \cap \rrpp^n$ for any $\xx_0 \in \rrpp^n$, and it is complex-balanced. 
\item Any positive steady state $\xx$ satisfies $\ln \xx - \ln \xx^* \in S^\perp$.
    
\item The function
    \eq{ 
        L(\xx) = \sum_{i=1}^n x_i(\ln x_i - \ln x^*_i - 1), 
    }
defined on $\rrpp^n$, is a strict Lyapunov function of \eqref{eq:poly-cor} within every invariant polyhedron $(\xx_0+S)\cap\rrpp^n$, with a global minimum at the corresponding complex-balanced steady state. 
\item Every positive steady state is locally asymptotically stable with respect to its invariant polyhedron. 
\end{enumerate}
\end{thm}

\begin{ex}
\label{ex:WR0} 
Consider the system of differential equations
\begin{equation}\label{eq:ex-WR0-ODE}
\begin{split}
     \frac{d x_1}{dt} &= - 12 x_1  +  x^2_3,
\\ \frac{d x_2}{dt} &= 14 x_1 - 4 x^2_2 + 8 x^2_3,
\\ \frac{d x_3}{dt} &= 10 x_1 + 4 x^2_2 - 10 x^2_3.
\end{split}
\end{equation}
We have $n = 3$ for the three state variables, and $m=3$ for the three distinct monomials. The matrices of source vertices and net direction vectors are 
\eq{ 
    \mm Y_s &= \begin{pmatrix} \yy_1 & \yy_2 & \yy_3 \end{pmatrix} = \begin{pmatrix} 1 & 0 & 0 \\
    0 & 2 & 0 \\ 
    0 & 0 & 2\end{pmatrix},   \\
    \mm W &= \begin{pmatrix} \vv w_1 & \vv w_2 & \vv w_3 \end{pmatrix} = \begin{pmatrix} -12 & \hphantom{-}0 & \hphantom{-}1  \\
    \hphantom{-}14 & -4 & \hphantom{-}8\\ 
    \hphantom{-}10 & \hphantom{-}4 & -10\end{pmatrix}
} 
respectively, which are inputs to \Cref{algorithm}. Let $V = \{ \yy_1, \yy_2,\yy_3\} \subset \zzp^n$. A generator for the cone $\ker \mm W \cap \rrp^m$ is $\vv c = (48/1441, 120/131, 576/1441)^\top$. In the notation of \Cref{algorithm}, $V_1 = [1,3]$.  Clearly $V$ is affinely independent and the net direction vectors admit the following unique decompositions: 
\eq{ 
    \vv w_1 &= 7 (\yy_2 - \yy_1) + 5(\yy_3 - \yy_1),  \\ 
    \vv w_2 &= 2(\yy_3 - \yy_2),  \\
    \vv w_3 &= (\yy_1 - \yy_3) + 4(\yy_2 - \yy_3).
} 
Therefore \eqref{eq:ex-WR0-ODE} admits a \WRz\ realization, whose weighted E-graph is shown in \Cref{fig:ex-WR0-realization}.

This implies that the system \eqref{eq:ex-WR0-ODE} has exactly one steady state within each invariant triangle given by $2x_1+x_2+x_3 = C$ for some $C>0$, and  this steady state is a global attractor within each such triangle. From \Cref{cor:alg}, we know the steady state set admits a monomial parametrization of the form $(a_1 s^2 , a_2 s, a_3s)$ for some constants $a_i  > 0$. In fact, the set of steady states is given by 
\eq{ 
    ( x_1^* , x_2^* , x_3^* )
    = \left(
        3 s^2 ,  \,
        \frac{\sqrt{330}}{2} s , \, 
        6 s
    \right),  
}
and an explanation for the coefficients above will be provided in \Cref{thm:eqm}.
\end{ex}

\begin{figure}[h!]
\centering 
\begin{subfigure}[b]{0.3\textwidth}
\centering 
    \begin{tikzpicture}[scale=1]
    \begin{axis}[
      view={115}{15},
      axis lines=center,
      width=6cm,height=6cm,
      ticks = none, 
      xmin=0,xmax=2.75,ymin=0,ymax=2.65,zmin=0,zmax=2.5,
    ]
    \draw [draw opacity=0, fill opacity=0.5, fill = pastelpink] (axis cs:1,0,0) -- (axis cs:0,2,0) -- (axis cs:0,0,2); 
    \addplot3 [no marks, dashed, gray!70] coordinates {(1,0,0)  (1,0,2.5)};
    \addplot3 [no marks, dashed, gray!70] coordinates {(2,0,0)  (2,0,2.5)};
    \addplot3 [no marks, dashed, gray!70] coordinates {(0,1,0)  (0,1,2.5)};
    \addplot3 [no marks, dashed, gray!70] coordinates {(0,2,0)  (0,2,2.5)};
    \addplot3 [no marks, dashed, gray!70] coordinates {(2.5,0,1)  (0,0,1) (0,2.5,1)};
    \addplot3 [no marks, dashed, gray!70] coordinates {(2.5,0,2)  (0,0,2) (0,2.5,2)};
    \addplot3 [no marks, dashed, gray!70] coordinates {(1,0,0)  (1,2.5,0) };
    \addplot3 [no marks, dashed, gray!70] coordinates {(2,0,0)  (2,2.5,0) };
    \addplot3 [no marks, dashed, gray!70] coordinates {(0,1,0)  (2.5,1,0) };
    \addplot3 [no marks, dashed, gray!70] coordinates {(0,2,0)  (2.5,2,0) };
    
    \addplot3 [only marks, blue] coordinates {(1,0,0) (0,2,0) (0,0,2)
    (0.3333,0.7778,0.5556) (0,1,1) (0.1,0.8,1)};
    
    \node [outer sep=1pt] (1) at (axis cs:1,0,0) {};
    \node [outer sep=1pt] (2) at (axis cs:0,2,0) {};
    \node [outer sep=1pt] (3) at (axis cs:0,0,2) {};
        \node at (1) [left] {$\yy_1$\,};
        \node at (2) [below=0.8pt] {$\yy_2$};
        \node at (3) [left] {$\yy_3$};
    \node (1t) at (axis cs:0.3333,0.7778,0.5556) {}; 
    \node (2t) at (axis cs:0,1,1) {}; 
    \node (3t) at (axis cs:0.1,0.8,1) {}; 
    \draw [-{stealth}, thick, blue, transform canvas={yshift=0pt}] (1)--(1t) node [midway, below right] {\!\!\ratecnst{$18$}}; 
    \draw [-{stealth}, thick, blue, transform canvas={yshift=0pt}] (2)--(2t) node [midway, above right] {\!\!\ratecnst{$4$}}; 
    \draw [-{stealth}, thick, blue, transform canvas={yshift=0pt}] (3)--(3t) node [midway, left] {\ratecnst{$10$}\,}; 
    \end{axis}
    \end{tikzpicture}
    \caption{}\label{fig:ex-WR0-newt}
\end{subfigure}\hspace{0.5cm}
\begin{subfigure}[b]{0.3\textwidth}
    \centering 
    \begin{tikzpicture}[scale=1]
    \begin{axis}[
      view={115}{15},
      axis lines=center,
      width=6cm,height=6cm,
      ticks = none, 
      xmin=0,xmax=2.75,ymin=0,ymax=2.65,zmin=0,zmax=2.5,
    ]
    \addplot3 [no marks, dashed, gray!70] coordinates {(1,0,0)  (1,0,2.5)};
    \addplot3 [no marks, dashed, gray!70] coordinates {(2,0,0)  (2,0,2.5)};
    \addplot3 [no marks, dashed, gray!70] coordinates {(0,1,0)  (0,1,2.5)};
    \addplot3 [no marks, dashed, gray!70] coordinates {(0,2,0)  (0,2,2.5)};
    \addplot3 [no marks, dashed, gray!70] coordinates {(2.5,0,1)  (0,0,1) (0,2.5,1)};
    \addplot3 [no marks, dashed, gray!70] coordinates {(2.5,0,2)  (0,0,2) (0,2.5,2)};
    \addplot3 [no marks, dashed, gray!70] coordinates {(1,0,0)  (1,2.5,0) };
    \addplot3 [no marks, dashed, gray!70] coordinates {(2,0,0)  (2,2.5,0) };
    \addplot3 [no marks, dashed, gray!70] coordinates {(0,1,0)  (2.5,1,0) };
    \addplot3 [no marks, dashed, gray!70] coordinates {(0,2,0)  (2.5,2,0) };
    
    \addplot3 [only marks, blue] coordinates {(1,0,0) (0,2,0) (0,0,2)};
    
    \node [outer sep=1pt] (1) at (axis cs:1,0,0) {};
    \node [outer sep=1pt] (2) at (axis cs:0,2,0) {};
    \node [outer sep=1pt] (3) at (axis cs:0,0,2) {};
        \node at (1) [left] {$\yy_1$\,};
        \node at (2) [below=0.8pt] {$\yy_2$};
        \node at (3) [left] {$\yy_3$};
    \draw [-{stealth}, thick, blue, transform canvas={yshift=0pt}] (1)--(2) node [midway, below] {\ratecnst{$7$}}; 
    \draw [-{stealth}, thick, blue, transform canvas={xshift=-0.25ex, yshift=0.1ex}] (1)--(3) node [midway, below left] {\ratecnst{$5$}}; 
    \draw [-{stealth}, thick, blue, transform canvas={xshift=0.25ex, yshift=-0.1ex}] (3)--(1) node [midway, below right] {\!\!\ratecnst{$1$}}; 
    \draw [-{stealth}, thick, blue, transform canvas={xshift=0.2ex, yshift=0.2ex}] (3)--(2) node [near end, above right] {\!\!\ratecnst{$4$}}; 
    \draw [-{stealth}, thick, blue, transform canvas={xshift=-0.2ex, yshift=-0.2ex}] (2)--(3) node [near start, left] {\ratecnst{$2$}\,}; 
    \end{axis}
    \end{tikzpicture}
    \caption{}\label{fig:ex-WR0-realization}
\end{subfigure}\hspace{0.5cm}
\begin{subfigure}[b]{0.3\textwidth}
    \centering 
    \begin{tikzpicture}[scale=1]
    \begin{axis}[
      view={115}{15},
      axis lines=center,
      width=6cm,height=6cm,
      ticks = none, 
      xmin=0,xmax=2.75,ymin=0,ymax=2.65,zmin=0,zmax=2.5,
    ]
    \draw [draw opacity=0, fill opacity=0.5, fill = pastelpink] (axis cs:1,0,0) -- (axis cs:0,2,0) -- (axis cs:0,0,2); 
    \addplot3 [no marks, dashed, gray!70] coordinates {(1,0,0)  (1,0,2.5)};
    \addplot3 [no marks, dashed, gray!70] coordinates {(2,0,0)  (2,0,2.5)};
    \addplot3 [no marks, dashed, gray!70] coordinates {(0,1,0)  (0,1,2.5)};
    \addplot3 [no marks, dashed, gray!70] coordinates {(0,2,0)  (0,2,2.5)};
    \addplot3 [no marks, dashed, gray!70] coordinates {(2.5,0,1)  (0,0,1) (0,2.5,1)};
    \addplot3 [no marks, dashed, gray!70] coordinates {(2.5,0,2)  (0,0,2) (0,2.5,2)};
    \addplot3 [no marks, dashed, gray!70] coordinates {(1,0,0)  (1,2.5,0) };
    \addplot3 [no marks, dashed, gray!70] coordinates {(2,0,0)  (2,2.5,0) };
    \addplot3 [no marks, dashed, gray!70] coordinates {(0,1,0)  (2.5,1,0) };
    \addplot3 [no marks, dashed, gray!70] coordinates {(0,2,0)  (2.5,2,0) };
    
    \addplot3 [only marks, blue] coordinates {(1,0,0) (0,2,0) (0,0,2) 
    (0.875,-0.5,0.75) (0,1,1) (0.1,0.8,1)};
    
    \node [outer sep=1pt] (1) at (axis cs:1,0,0) {};
    \node [outer sep=1pt] (2) at (axis cs:0,2,0) {};
    \node [outer sep=1pt] (3) at (axis cs:0,0,2) {};
        \node at (1) [left] {$\yy_1$\,};
        \node at (2) [below=0.8pt] {$\yy_2$};
        \node at (3) [left] {$\yy_3$};
    \node (1t) at (axis cs:0.875,-0.5,0.75) {}; 
    \node (2t) at (axis cs:0,1,1) {}; 
    \node (3t) at (axis cs:0.1,0.8,1) {}; 
    \draw [-{stealth}, thick, blue, transform canvas={yshift=0pt}] (1)--(1t) node [midway, above right] {\!\ratecnst{$4$}}; 
    \draw [-{stealth}, thick, blue, transform canvas={yshift=0pt}] (2)--(2t) node [midway,  right] {\ratecnst{$4$}}; 
    \draw [-{stealth}, thick, blue, transform canvas={yshift=0pt}] (3)--(3t) node [midway, below left] {\ratecnst{$10$}\!\!}; 
    \end{axis}
    \end{tikzpicture}
    \caption{}\label{fig:ex-not-WR0}
\end{subfigure}
\caption{(a) A weighted E-graph realizing \eqref{eq:ex-WR0-ODE} from \Cref{ex:WR0}, which admits 
	(b) a \WRz\ realization. 
	(c) A weighted E-graph realizing \eqref{eq:ex-WR0-not-ODE} from \Cref{ex:not-WR0} that does not admit a \WRz\ realization. }
\label{fig:ex-frompoly}
\end{figure}
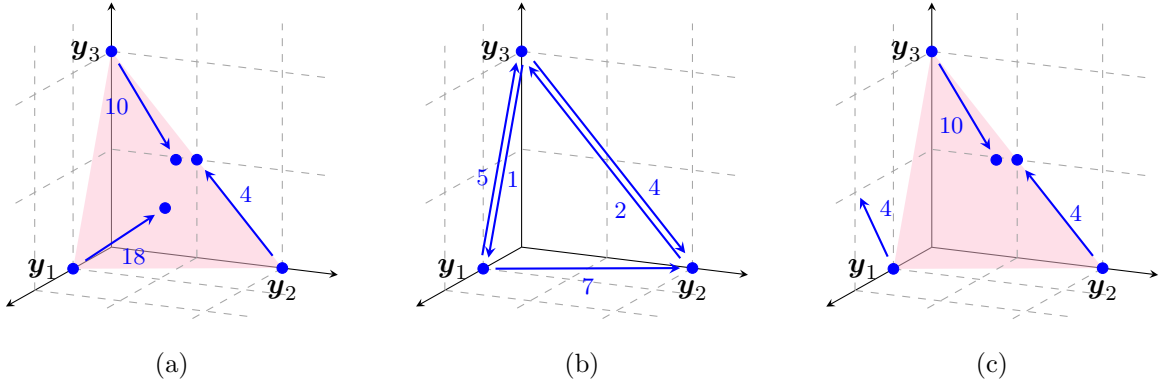

\begin{ex}
\label{ex:not-WR0}
Consider the system of differential equations
\begin{equation}\label{eq:ex-WR0-not-ODE}
\begin{split}
     \frac{d x_1}{dt} &= -\frac{1}{2} x_1  +   x^2_3,
\\ \frac{d x_2}{dt} &= -2 x_1 -4 x^2_2 + 8 x^2_3,
\\ \frac{d x_3}{dt} &= 3 x_1 + 4 x^2_2  -10 x^2_3.
\end{split}
\end{equation}
Again, we have $n = 3$  and $m=3$. The monomials are the same as those in the previous example. The difference lies in the first column of the matrix of net direction vectors  
\eq{ 
    \mm W &= \begin{pmatrix} \vv w_1 & \vv w_2 & \vv w_3 \end{pmatrix} = 
    \begin{pmatrix} 
    -\frac{1}{2} & \hphantom{-}0 & \hphantom{-}1  \\
    -2 & -4 & \hphantom{-}8 \\ 
    \hphantom{-}3 & \hphantom{-}4 & -10
    \end{pmatrix},  
} 
whose kernel is spanned by $\vv c = (2,1,1)^\top$.  As in the previous example, the vertices $\yy_1$, $\yy_2$, and $\yy_3$ are affinely independent. However,  $\vv w_1 \not\in \mrm{Cone}\{\yy_j - \yy_1 \st j = 2,3\}$, so no \WRz\ realization exists. 
\end{ex}

\subsection{The set of positive steady states of a \WRz\ realization}
\label{sec:ss-set}

\Cref{algorithm} determines whether a given polynomial dynamical system admits a \WRz\ realization. If it does, its steady state set is in fact log-linear. In this section, we write down a system of linear equations whose solution set is in bijection with the set of positive steady states; this provides an explicit parametrization of the set of positive steady states. 

For any $\vv z \in \rr^n$ and $\xx \in \rrpp^n$, define the component-wise operations $\exp \vv z = (e^{z_1}, e^{z_2},\ldots, e^{z_n})^\top$ and $\log(\xx) = (\log x_1, \log x_2,\ldots, \log x_n)^\top$. We extend these operations to sets. If $Z \subseteq \rr^n$, then $\exp(Z) = \{ \exp \vv z \st \vv z \in Z\}$, and if $X \subseteq \rrpp^n$, then $\log(X) = \{ \log \xx \st \xx \in X\}$. 

Assume that the polynomial dynamical system 
    \eqn{ \label{eq:alg-ODE} 
        \frac{d\xx}{dt} = \sum_{i=1}^m \xx^{\yy_i} \vv w_i, 
    }
with distinct $\yy_i \in \zzp^n$ and $\vv w_i \in \rr^n \setminus\{\vv 0\}$, passes \Cref{algorithm}, i.e., it admits a \WRz\ realization $(V,E,\vv\kk)$. Without loss of generality, assume the vertices are ordered according to connected components in $(V,E)$, i.e., the first $m_1$ vertices belong to the connected component $(V_1,E_1)$, the next $m_2$ vertices belong to the connected component $(V_2,E_2)$, and so forth. Let $\{ \vv c_1, \vv c_2,\ldots, \vv c_\ell\}$ be a minimal set of generators of $\ker \mm W \cap \rrp^m$, ordered in an analogous way.  From \Cref{algorithm}, we know that the supports of the vectors $\vv c_1$, $\vv c_2,\ldots, \vv c_\ell$ correspond to the connected components of $(V,E)$.  

Let $\vv c_1 = (\alpha_1, \alpha_2,\ldots, \alpha_{m_1}, 0,\ldots, 0)^\top$. Define matrix $\mm D_1 \in \rr^{(m_1-1) \times n}$ whose rows are the affine vectors from $\yy_1$ to the remaining vertices of $V_1$, and define vector $\vv J_1  \in  \rr^{m_1-1}$ using the log-differences of the components of $\vv c_1$, i.e., 
\eq{ 
    \mm D_1 = \begin{pmatrix}
        \yy_2 - \yy_1 \\
        \yy_3 - \yy_1 \\
        \vdots \\
        \yy_{m_1} - \yy_1
    \end{pmatrix} 
    \quad \text{and} \quad 
    \vv J_1 = \begin{pmatrix}
    \log (\alpha_2/ \alpha_1 )\\
    \log (\alpha_3 /  \alpha_1) \\
    \vdots  \\
    \log (\alpha_{m_1} / \alpha_1 )
    \end{pmatrix} . 
}
For the connected component $(V_p,E_p)$, define $\mm D_p$ and $\vv J_p$ in a similar fashion. Define  
\eqn{ \label{eq:pf-mtx} 
    \mm D = \begin{pmatrix}
    \mm D_1 \\ \hline 
    \mm D_2 \\ \hline 
    \vdots \\ \hline 
    \mm D_{\ell} 
    \end{pmatrix} \in \rr^{(m-\ell) \times n}
    \quad \text{and} \quad 
    \vv J = \begin{pmatrix}
        \vv J_1 \\ \hline 
        \vv J_2 \\ \hline 
        \vdots \\ \hline 
        \vv J_\ell
    \end{pmatrix} \in \rr^{m-\ell} .
}

\begin{thm}
\label{thm:eqm}
Suppose the system of differential equation \eqref{eq:alg-ODE} admits a \WRz\ realization $(V,E,\vv\kk)$, and let $\mm D \in \rr^{(m-\ell)\times n}$ and $\vv J \in \rrpp^{m-\ell}$ be defined as in \eqref{eq:pf-mtx}. Then the system $\mm D \vv z = \vv J$ is solvable. Let $\vv z^* + \ker \mm D$ be its solution set. Then the set of positive steady states of \eqref{eq:alg-ODE} is $\exp(\vv z^* + \ker \mm D)$. 
\end{thm}
\begin{proof}
    First we prove that the linear system $\mm D \vv z = \vv J$ is solvable. Consider $\mm D_1$. The vertices $\yy_1$, $\yy_2,\ldots, \yy_{m_1}$ in the first connected component are affinely independent, so the rows of $\mm D_1$ are linearly independent. Moreover, as noted in \Cref{rmk:stoich-linkage} the row-space of $\mm D_1$ is the associated linear subspace $S(V_1)$. Therefore $\rank \mm D_1 = m_1 - 1$, and the matrix $\mm D_1$ is surjective onto $\rr^{m_1-1}$. 
    
    Similarly, for each $p=2,\ldots, \ell$, the row-space of the matrix $\mm D_p$ is  $S(V_p)$, and the matrix $\mm D_p$ is surjective. In addition, because the realization $(V,E,\vv\kk)$ has deficiency zero, $S(V_1)$, $S(V_2),\ldots, S(V_\ell)$ are linearly independent; in other words, the $m-\ell$ rows of the matrix $\mm D$  are linearly independent. Consequently,  $\mm D$ is surjective, and the system $\mm D \vv z = \vv J$ is solvable. 

    Let $\vv z^* + \ker \mm D$ be the set of solution to  $\mm D \vv z = \vv J$. We next show that each solution can be related to a positive steady state of \eqref{eq:alg-ODE}, which by definition satisfies 
    \eq{ 
        \vv 0 =  \sum_{i=1}^m \xx^{\yy_i} \vv w_i  .
    }
    In other words, $(\xx^{\yy_1}, \ldots, \xx^{\yy_m})^\top$ lies in the steady state flux cone $\ker \mm W \cap \rrpp^m$. Decomposing this vector with respect to the generators of the cone allows us to focus on one connected component at a time. 
    
    For simplicity of notation, consider the first connected component. At steady state, for some constant $\lambda >0$, we have $\xx^{\yy_j} = \lambda \alpha_j$ for $j=1$, $2,\ldots, m_1$. Thus 
    \eq{ 
        \xx^{\yy_j - \yy_1} = \frac{\alpha_j}{\alpha_1}
    }
    for $j=2,3,\ldots, m_1$. Taking the logarithm of both sides, we obtain the system $\mm D_1 \vv z = \vv J_1$ with $\vv z = \log \xx$. 

    Repeating this computation for each connected component, we conclude that $\xx$ is a positive steady state for \eqref{eq:alg-ODE} if and only if $\xx$ solves $\mm D \vv z = \vv J$ with $\vv z = \log \vv x$. This leads us to the characterization of the set of positive steady states for \eqref{eq:alg-ODE} as $\exp(\vv z^* + \ker \mm D)$, where $\vv z^* + \ker \mm D$ is the set of solutions to $\mm D \vv z = \vv J$. 
\end{proof}

\subsection{Extension to polynomial systems with unspecified coefficients}
\label{sec:unspec_coeff}

If instead of \eqref{eq:alg-ODE}, we need to analyze 
\eqn{ \label{eq:alg-vark-ODE} 
        \frac{d\xx}{dt} = \sum_{i=1}^m a_i \xx^{\yy_i} \vv w_i 
    }
for some unknown $a_i > 0$, it turns out that the answer as to whether a \WRz\ realization exists is the same: 

\begin{thm}
\label{thm:vark}
    For any $a_i > 0$, the system \eqref{eq:alg-vark-ODE}   admits a \WRz\ realization $(V,E,\vv\kk)$ if and only if the system \eqref{eq:alg-ODE} admits a \WRz\ realization $(V,E,\vv\kk^*)$. Moreover, $\kk_{ij} = a_i \kk^*_{ij}$. 
\end{thm}
\begin{proof}  
    The forward implication is trivial. We focus our attention on the other direction. 
    For any $i$, $j$, let $\kk_{ij} = a_i \kk^*_{ij}$, so $\kk_{ij} > 0$ if and only if $\kk^*_{ij} > 0$. In other words, the weighted E-graph $(V,E,\vv\kk)$ shares the same set of edges as $(V,E,\vv\kk^*)$. Because the deficiency is characterized by affine and linear independence of the connected components, and the two graphs share the same structure, $(V,E,\vv\kk)$ is weakly reversible and deficiency zero if and only if $(V,E,\vv\kk^*)$ is. 
    
    Suppose $(V,E,\vv\kk^*)$ is a realization of \eqref{eq:alg-ODE}. Then in $(V,E,\vv\kk)$, the net direction vector from $\yy_i$ can be expanded using the realization $(V,E,\vv\kk^*)$, since 
    \eq{
         a_i \vv w_i  = a_i\sum_{(i,j) \in E}  \kk^*_{ij} (\yy_j - \yy_i) = \sum_{(i,j) \in E} \kk_{ij} (\yy_j - \yy_i) . 
    }
    Therefore, $(V,E,\vv\kk)$ realizes \eqref{eq:alg-vark-ODE}. 
\end{proof}

\subsection{Deficiency zero realizations that are not weakly reversible}

If a polynomial dynamical system admits a deficiency zero realization that is {\em not} weakly reversible, then its dynamics is also greatly restricted: it can have no positive steady states, no oscillations, and no chaotic dynamics~\cite{horn1972necessary, Feinberg1972, feinberg2019foundations}. Actually, such realizations are special examples of mass-action system that are {\em not consistent}~\cite{AndersonBrunnerCraciunJohnston2020}. An E-graph $(V,E)$ is said to be {\em consistent} if there exist real numbers $\alpha_{ij}>0$ such that 
\eq{ 
    \sum_{(i,j) \in E} \alpha_{ij}(\yy_j - \yy_i)  = \vv 0.
}
It is easy to see that a polynomial dynamical system of the form \eqref{eq:alg-ODE}  has a realization that is not consistent if and only if 
\begin{equation}
\label{inconsistent}
    \ker \mm W \cap \mathbb{R}^m_{>} = \emptyset.
\end{equation}
If a polynomial dynamical system has a realization that is not consistent, then it cannot have realization that is weakly reversible, because weakly reversible systems must have at least one positive steady state~\cite{Boros2019}. 
Therefore, if \Cref{algorithm} is accompanied by a preprocessing step that checks condition~\eqref{inconsistent}, then that step will  decide whether our  given system \eqref{eq:alg-ODE}  has a realization that is not consistent; in particular, this step will also find all cases where our given  system has a deficiency zero realization that is {\em not} weakly reversible.

\section*{Acknowledgements} 

The authors were supported in part by the National Science Foundation under grants 
DMS--1816238 and DMS--2051568. G.C. was also partially supported by the Simons Foundation.  
\bibliographystyle{siam}
\bibliography{cit}

\end{document}